\documentclass[review]{elsarticle}

\usepackage{amssymb}
\usepackage{amsthm}
\usepackage{lineno,hyperref}
\modulolinenumbers[5]

\journal{Journal of \LaTeX\ Templates}

\newtheorem{thm}{Theorem}

\newtheorem{lem}[thm]{Lemma}
\newdefinition{rmk}{Remark}
\newproof{pf}{Proof}
\newproof{pot}{Proof of Theorem \ref{mainresult}}

\begin{document}

\begin{frontmatter}

\title{The Hamilton-Waterloo Problem with $C_4$ and $C_m$ Factors\tnoteref{mytitlenote}}
\tnotetext[mytitlenote]{This work is supported by The Scientific and Technological Research Council of Turkey (T\"UB\.{I}TAK) under project number 113F033.}


\address[mymainaddress]{Department of Engineering Sciences, Faculty of Engineering, Istanbul University, Avcilar Campus,  Avcilar - ISTANBUL, 34320 TURKEY}
\address[mysecondaryaddress]{Department of Mathematics, Gebze Technical University, Cayirova Campus, Gebze - Kocaeli, 41400 TURKEY}

\author[mymainaddress]{U\u{g}ur Odaba\c{s}{\i} }
\ead{ugur.odabasi@istanbul.edu.tr}
\author[mysecondaryaddress]{Sibel \"Ozkan \corref{mycorrespondingauthor}}
\cortext[mycorrespondingauthor]{Corresponding author}
\ead{s.ozkan@gtu.edu.tr}

\begin{abstract}
The Hamilton-Waterloo problem with uniform cycle sizes asks for a $2-$ factorization of the complete graph $K_v$ (for odd {\em v}) or $K_v$ minus a $1-$factor (for even {\em v}) where $r$ of the factors consist of $n-$cycles and $s$ of the factors consist of $m-$cycles with $r+s=\left \lfloor \frac{v-1}{2} \right \rfloor$. In this paper, the Hamilton-Waterloo Problem with $4-$cycle and $m-$cycle factors for odd $m\geq 3$ is studied and all possible solutions with a few possible exceptions are determined.
\end{abstract}

\begin{keyword}
\texttt 2-factorizations\sep Hamilton-Waterloo Problem\sep Oberwolfach Problem \sep Resovable decompositions \sep Cycle decompositions
\end{keyword}

\end{frontmatter}


\section{Introduction}
A {\em decomposition} of a graph {\em G} is a set $\mathcal{H}=\{H_1,H_2,\dots,H_k\}$ of edge-disjoint subgraphs of {\em G} such that $\bigcup\limits_{i=1}^k E(H_i)=E(G)$. A {\em H-decomposition} is a decomposition of {\em G} such that  $H_i \cong H$ for all $H_i \in \mathcal{H}$. If each $H_i$ is a cycle (or union of cycles), then $\mathcal{H}$ is called a {\em cycle decomposition}. A $\{F_1^{k_1},F_2^{k_2}, \dots ,F_l^{k_l} \}-${\em factorization} of a graph {\em G} is a decomposition which consists precisely of $k_i$ factors isomorphic to $F_i$.

The case $H\cong K_2$ is known as {\em 1-factorization}. Another important case is {\em 2-factorization} where every vertex in the graph $H$ has degree $2$. Whether there exists a $2-$factorization of $K_v$ with prescribed $2-$factors is a long standing important problem in combinatorial design theory.

One of the $2-$factorization problems is the  {\em Oberwolfach Problem} which was first formulated by Ringel at an Oberwolfach meeting in 1967 and is related to the possible seating arrangements at the conference. The problem was inspired by a question whether {\em v} mathematicians could be seated in such a way that each mathematician sits next to each other mathematician exactly once over $\left \lfloor \frac{v-1}{2} \right \rfloor$ days, where there are $k_i$ round tables with $m_i$ seats for $ 1\leq i \leq t$ satisfying $\sum\limits_{i=1}^{t}k_im_i=v$. In graph theory language, the problem asks for a $2-$factorization of the complete graph $K_v$ (or for even $v$, $2-$factorization of $K_v-I$) into  $2-$factors each of which is isomorphic to a given $2-$factor {\em H}. If {\em H} consists of $k_i$  $m_i$-cycles, $1 \leq i\leq t$,  then the corresponding {\em Oberwolfach problem} is denoted by $OP(m_1^{k_1},m_2^{k_2}, \dots ,m_t^{k_t})$.

It is known that the solutions to the cases $OP(3^2), OP(3^4), OP(4, 5)$  and $OP(3^2, 5)$ do not exist \cite{alspach, kohler, piotrowski german unpublished paper}. The Oberwolfach Problem for a single cycle size $OP(m^k)$ for all $m\geq3$ has been solved in two separate cases: odd cycles in \cite{alspach} and the even cycle case in \cite{hoffman}.

A generalization of the Oberwolfach Problem is the {\em Hamilton-Waterloo Problem} where the conference takes places in two venues; Hamilton and Waterloo, the first of which has {\em k} round tables, each seating $n_i$ people for $i = 1, \dots ,k$ , the second of which has {\em l} round tables each seating  $m_i$ people for $i = 1, \dots ,l$ (necessarily $\sum\limits_{i=1}^{k}n_i=\sum\limits_{i=1}^{l}m_i=v$).

If we let $n = n_1 = n_2 = \dots = n_k$ and $m = m_1 = m_2 = \dots = m_l$, then each $2-$factor is composed of either $n$-cycles or {\em m}-cycles. This version of the Hamilton-Waterloo Problem, with uniform cycle sizes, has attracted most of the attention and we use the notation to denote the problem with $r$ factors of $n-$cycles and $s$ factors of $m-$cycles by $(n, m)-$HWP$(v; r, s)$. The obvious necessary conditions for the existence of a solution to $(n, m)-$HWP$(v; r, s)$ are given by Adams et al. in \cite{adams}:

\begin{lem} \textbf{\emph{\cite{adams}}} \label{lemma:necessary}
Let v, n, m, r and s be non-negative integers with $n,m\geq 3$. If there exists a solution to $(n, m)-${\em HWP}$(v; r, s)$, then
\begin{itemize}
\item[1)] if $r > 0$,  $v \equiv 0 \bmod{n}$,
\item[2)] if  $s > 0$, $v \equiv 0 \bmod{m}$,
\item[3)] $r+s=\left \lfloor \frac{v-1}{2} \right \rfloor$.
\end{itemize}
\end{lem}

The first result on the Hamilton-Waterloo Problem is settled by Adams et al. \cite{adams} in $2002$. They solved the cases $(n,m)\in \{(4,6),(4,8),(4,16),(8,16),(3,5),$ $(3,15),(5,15)\}$ and settled the problem for all $v\leq 16$. With a few possible exceptions  when $m=24$ and $48$, Danziger et al. \cite{danziger} solved the problem for the case $(n,m)=(3,4)$.

The case $n=3$ and $m=v$, i.e. triangle-factors and Hamilton cycles, has attracted much attention and remarkable progresses are obtained by Horak et al. \cite{horak}, Dinitz and Ling \cite{dinitz1, dinitz2}. In \cite{bryant}, Bryant et al. have settled the Hamilton-Waterloo Problem for bipartite $2-$factors, and in \cite{buratti} Buratti and Rinaldi studied regular 2-factorizations leading to some cyclic solutions to Oberwolfach and Hamilton-Waterloo Problems, and also in \cite{buratti1}, an infinite class of cyclic solutions to the Hamilton-Waterloo Problem is given.

Fu and Huang \cite{fu} solved the case of $4-$cycles and $m-$cycles for even {\em m}, and also settled all cases where $m=2n$ and {\em n} is even in 2008. Two years later Keranen and \"Ozkan \cite{sibel} solved the case of $4-$cycles and a single factor of  $m-$cycles for odd {\em m}.

Most of the results involve the cases of even cycles or the cycles of same parity. Solving the Hamilton-Waterloo Problem for cycles with different parity is a more difficult problem and is not studied much.

Here we consider $4-$cycle and odd cycle factors, and complete the remaining cases in \cite{sibel}. Our result also complements the results of Fu and Huang \cite{fu} and shows that the necessary conditions are sufficient also for odd $m$ with a few exceptions. Here is our main result.

\begin{thm} \label{mainresult}
For all positive integers r, s and odd $m\geq 3$, a solution to $(4, m)-${\em HWP}$(v; r, s)$ exists if and only if $4|v$, $m|v$ and $r+s=\frac{v-2}{2}$ except  possibly when $r=2$ and $v=8m$ or $v=24,48$ when $m=3$ and $r=6$.
\end{thm}

\section{Preliminary Results}

If $G_1$ and $G_2$ are two edge disjoint graphs on the same vertex set, then  $G_1\oplus G_2$ will denote the graph on the same vertex set with $E(G_1\oplus G_2)=E(G_1)\cup E(G_2)$. Also $\alpha G$ will denote the vertex disjoint union of the $\alpha$ copies of {\em G}.

We will denote a {\em complete equipartite graph} of {\em b} parts of size {\em a} each by $K_{a:b}$. $K_{a:2}$ is called {\em complete bipartite graph} and denoted by $K_{a,a}$ as well.

Liu \cite{liu} gave a complete solution to the Oberwolfach Problem for complete equipartite graphs where all cycles have the same length and we will use this result in our main construction.

\begin{thm} \textbf{\emph{\cite{liu}}} \label{liu}
The complete equipartite graph $K_{a:b}$ has a $C_l-$factorization for  $l \geq 3$ and $a\geq 2$ if and only if  $l|ab$, $a(b-1)$ is even, l is even if $b=2$ and $(a,b,l)\neq (2,3,3),(6,3,3)$, $(2,6,3),(6,2,6)$.
\end{thm}

Let $H$ be a finite additive group and let $S$ be a subset of $H-\{0\}$ such that the inverse of every element of $S$ also belongs to $S$. The {\em Cayley graph} over $H$ with connection set $S$, denoted by $Cay(H,S)$, is the graph with vertex set $H$ and edge set $E(Cay(H,S))=\{(a,b)|a,b\in H, a-b\in S\}$. Note that, since $ S=S^{-1} $, here $Cay(H,S)$ is not directed.

Let {\em G} be graph and $G_0, G_1, G_2, \dots , G_{k-1}$ be vertex disjoint copies of {\em G} with $v_i \in V(G_{i})$ for each $v \in V(G)$. Then the graph $G[k]$ is a graph with vertex set $V(G[k])=V(G_0)\cup V(G_1)\cup V(G_2)\cup \dots \cup V(G_{k-1})$  and edge set
$E(G[k])=\{u_iv_j: uv \in E(G) \quad \mbox{and} \quad 0\leq i,j \leq k-1\}$. For example $K_m[2]\cong K_{2m}-I$ and $K_2[m]\cong K_{m,m}$ where $I$ is a $1-$factor of $K_{2m}$.

It is easy to see that if a graph {\em G} has an {\em H}$-$decomposition, then there exists an $H[k]-$ decomposition of $G[k]$. Moreover if a graph {\em G} has an $H-$factorization, then there exists an $H[k]-$factorization of $G[k]$.

In fact, this graph operation is a generalization of H$\ddot{a}$ggkvist's {\em doubling construction} and it coincides with a special case of a graph product called lexicographic product. H$\ddot{a}$ggkvist \cite{haggkvist} constructed $2-$factorizations containing even cycles using $G[2]$.

\begin{lem} \textbf{\emph{\cite{haggkvist}}} \label{lemma:Haggvist's lemma}
Let $G$ be a path or a cycle with $m$ edges and let $H$ be a $2-$ regular graph on $2m$ vertices where each component of $H$ is a cycle of even length. Then $G[2]$ has an $H-$decomposition.
\end{lem}

Baranyai and Szasz \cite{baranyai} have shown that if {\em G} consists of $x$ Hamilton cycles and if {\em H} has $y$ vertices and consists of $z$ Hamilton cycles then their lexicographic product is decomposable into $xy+z$ Hamilton cycles. So, $C_m[n]$ has a $C_{mn}-$factorization. Also Alspach et al. \cite{alspach} have shown that for an odd integer {\em m} and a prime {\em p} with $3\leq m\leq p$, $C_m[p]$ has a $C_p-$factorization.

By \cite{alspach} and \cite{hoffman}, solutions to $OP(4^{v/4})$ and  $OP(m^{v/m})$ exist except $m=3$ and $v=6$ or $v=12$. That is a solution to $(4, m)-$HWP$(v; r, s)$ exists for $r=0$ or $s=0$ with exceptions $(v,m,r)=(6,3,0)$ and $(v,m,r)=(12,3,0)$.  So, we can assume that $r\neq 0$ and $s\neq 0$.

In our case $4|v$,  $m|v$ and $m$ is odd. Then there exists a $t\in \mathbb{Z^+}$ such that $v=4mt$.

Note that;

\begin{equation} \label{one}
K_{4mt}\cong K_{mt}[4]\oplus mt K_4
\end{equation}
 or equivalently
  \begin{equation} \label{two}
 K_{4mt}-I\cong K_{mt}[4]\oplus mt C_4
 \end{equation}
where $V(K_{4mt})=V(K_{mt}[4])$ and $I$ is a $1-$factor in $K_{4mt}$. So $K_{4mt}$ has a $\{(C_m[4])^{(mt-1)/2},K_4\}-$factorization for odd {\em t} by (\ref{one}) and, $K_{mt}$ has a $C_m-$ factorization for odd {\em t} by \cite{alspach}. In short, for odd {\em t} we have
\begin{equation} \label{three}
K_{4mt} \cong \overbrace{tC_m[4]\oplus tC_m[4]\oplus \dots \oplus tC_m[4] }^{(mt-1)/2} \oplus mtK_4.
\end{equation}

Similarly, $K_{4mt}$ has a  $\{(C_m[4])^{^{(mt-2)/2}}, K_{4,4},K_4\}-$factorization for even {\em t} by (\ref{one}) and, $K_{mt}$ has a $\{C_m^{(mt-2)/2},K_2\}-$factorization for even {\em t} by \cite{hoffman}. In short, for even {\em t}, we have
\begin{equation} \label{four}
K_{4mt} \cong \overbrace{tC_m[4]\oplus tC_m[4]\oplus \dots\oplus tC_m[4] }^{(mt-2)/2} \oplus \frac{mt}{2}K_{4,4} \oplus mtK_4.
\end{equation}
with exceptions  $m=3$ and $t=2$ or $t=4$.

In our proofs, we will use these decompositions with appropriate factorizations of  $C_m[4]$'s.

It is obvious that a $2-$factorization of $C_m[4]$ has exactly four factors.  The followings will be shown:
\begin{itemize}
\item[(i)] $C_m[4]$ has a $C_4-$factorization (Lemma \ref{lemma C4-factorization}),
\item[(ii)] $C_m[4]$ has a $C_m-$factorization (Lemma \ref{lemma Cm-factorization}),
\item[(iii)] $C_m[4]$ has a $\{C_4^2,C_m^2\}-$factorization (Lemma \ref{lemma C4^2,Cm^2-factorization}),
\item[(iv)] $C_m[4]$ has no $\{C_4^1,C_m^3\}-$factorization for odd $m$ (Lemma \ref{lemma C4^1,Cm^3-factorization}).
\end{itemize}

\begin{lem} \label{lemma C4-factorization}
For every integer $m\geq 3$, $C_m[4]$ has a $C_4-$factorization.
\end{lem}

\begin{proof}
Note that $C_m[4]\cong C_m[2][2]$. By Lemma \ref{lemma:Haggvist's lemma}, $C_m[2]$ can be decomposed into $C_{2m}-$factors, and each $C_{2m}$ can be decomposed into two $1-$factors. So $C_m[2]$ has a $1-$factorization. If $F$ is a $1-$factor in $C_m[2]$, $F[2]$ is a $C_4-$factor in $C_m[4]$ since $K_2[2]\cong C_4$. Hence $C_m[4]$ has a $C_4-$factorization.
\end{proof}

\begin{lem} \label{lemma Cm-factorization}
For every integer $m\geq 3$, $C_m[4]$ has a $C_m-$factorization.
\end{lem}

\begin{proof}
We can represent $C_m[4]$ as the Cayley graph over $V_4\mathbb{\times Z}_m$ with connection set $V_4\times \{1,-1\}$ where $V_4$ is the additive group of $\mathbb{F}_4=\{0,1,x,x^2\}$. Let $C=(v_0,v_1, \dots ,v_{m-1})$ be an $m-$cycle of $C_m[4]$ where $v_i=(x^i,i)$ for $0\leq i \leq m-1$. In the case of  $m\equiv1 \pmod{3}$ replace $v_{m-1}$ with $(x,m-1)$. It can be checked that
\begin{center}
$F= C \cup (x,1) \cdot C \cup (x^2,1)\cdot C \cup (0,1) \cdot C$
\end{center}
is a $2-$ factor of $C_m[4]$. It is also easy to check that
\begin{center}
$\mathcal{F}=\{ F, F+(1,0), F+(x,0), F+(x^2,0)\}$
\end{center}
is a $2-$ factorization of $C_m[4]$.
\end{proof}

It is evident that the addition by $(1,0)$ and multiplication by $(x,1)$ are automorphisms of the above factorization $\mathcal{F}$. These automorphisms clearly generate $AGL(1,4)$ (the $1-$dimensional affine general linear group over $\mathbb{F}_4$).

\begin{lem} \label{lemma C4^2,Cm^2-factorization}
For every integer $m\geq 3$, $C_m[4]$ has a $\{C_4^2,C_m^2\}-$factorization.
\end{lem}

\begin{proof}
We can represent $C_m[4]$ as the Cayley graph $\Gamma$ over $\mathbb{Z}_4\mathbb{\times Z}_m$ with connection set $\mathbb{Z}_4\times \{1,-1\}$.

When $m$ is even, let $C=(v_0,v_1, \dots ,v_{m-1})$ and $C'=(v'_0,v'_1, \dots ,v'_{m-1})$ be the $m-$cycles of $\Gamma$ where $v_i=(2i,i)$ and $v'_i=(0,i)$  for $0\leq i \leq m-1$. Then
\begin{center}
$F_1= C \cup (C+(1,0)) \cup (C+(2,0))\cup (C+(3,0))$  and\\
$F'_1= C' \cup (C'+(1,0)) \cup (C'+(2,0))\cup (C'+(3,0))$
\end{center}
are two edge-disjoint $m$-cycle factors of $\Gamma$.

Also let $C_*=((0,1),(1,0),(2,1),(3,0))$ be a $4-$cycle of $\Gamma$. Then
\begin{center}
$F_2= \bigcup\limits_{i=0}^{m-1}(C_*+(0,i))$
\end{center}
is a $4$-cycle factor of $\Gamma$. It can be checked that
\begin{center}
$\mathcal{F}=\{ F_1, F'_1, F_2, F_2+(1,0)\}$
\end{center}
is a $2-$ factorization of $\Gamma$.

When $m$ is odd, let $C$, $C'$ and $C_*$ be defined as above with $v_{m-1}=(1,m-1)$. Also let $C_{*}^{'}=((0,0),(2,m-1),(1,m-2),(3,m-1))$ be a $4-$cycle of $\Gamma$. Then
\begin{center}
$F_1= C \cup (C+(1,0)) \cup (C+(2,0))\cup (C+(3,0))$ \\
$F'_1= C' \cup (C'+(1,0)) \cup (C'+(2,0))\cup (C'+(3,0))$ and \\
$F_2= \bigcup\limits_{i=0}^{m-3}(C_*+(0,i)) \cup C_{*}^{'} \cup (C_{*}^{'}+(2,0))$
\end{center}
are $2-$factors of $\Gamma$. It can be checked that
\begin{center}
$\mathcal{F}=\{ F_1, F'_1, F_2, F_2+(1,0)\}$
\end{center}
is a $2-$ factorization of $\Gamma$.
\end{proof}

\begin{lem} \label{lemma C4^1,Cm^3-factorization}
For any odd integer $m\geq 3$, $C_m[4]$ has no  $\{C_4^1,C_m^3\}-$factorization; that is, $C_m[4] \not\cong mC_4\oplus 4C_m\oplus 4C_m\oplus 4C_m$.
\end{lem}

\begin{proof}
Consider $C_m[4]$ as the Cayley graph $\Gamma$ over $\mathbb{Z}_4\mathbb{\times Z}_m$ with connection set $\mathbb{Z}_4\times \{1,-1\}$ as before.

We prove the Lemma by contradiction. So assume that $\Gamma$ can be decomposed into three $C_m-$factors and a single $C_4-$factor.

Since {\em m} is odd, each $m-$cycle in $\Gamma$ contains one and only one vertex $(a,i)$ of $\Gamma$ for each $i \in \mathbb{Z}_m $. When we remove the three $C_m-$factors, we are left with a $2-$regular graph where each vertex $(a,i)$ is adjacent to only one vertex  $(b,i-1)$ and  only one vertex $(c,i+1)$ for some $b,c \in \mathbb{Z}_4$. So,  this  $2-$regular graph can not contain any $4-$cycles.

Hence, $\Gamma$ has no $\{C_4^1,C_m^3\}-$factorization.
\end{proof}

\section{When {\em r} is odd}
Now, we can prove that for odd $m\geq 3$, a solution to $(4, m)-$HWP$(v; r, s)$ exists for all odd {\em r} (or even {\em s}) satisfying the necessary conditions.

\begin{thm} \label{thm:oddr}
For all positive odd integers r and $m\geq 3$, a solution to \\ $(4, m)-${\em HWP}$(v; r, s)$ exists if and only if $4|v$, $m|v$ and $r+s=\frac{v-2}{2}$ except possibly  $v=24,48$ when $m= 3$.
\end{thm}

\begin{proof}
If a solution to $(4, m)-$HWP$(v; r, s)$ exists, then by Lemma \ref{lemma:necessary},  $m|v$, $4|v$ and $r+s=\frac{v-2}{2}$ since {\em v} is even.

For the sufficiency part, assume  $m\geq 3$ is odd, $m|v$ and $4|v$. Then, since $\gcd(4,m)=1$, $4m|v$. Thus, there exists  an integer {\em t} such that $v=4mt$.

We will prove the theorem in two cases; $t$ is odd or even.

\noindent \textbf{Case 1:} Assume $t$ is odd.

By (\ref{three}), $K_{4mt}-I$ has a  $\{(C_m[4])^{^{(mt-1)/2}},C_4\}-$factorization. Now, let $r_1,s_1$ and $x$ be non-negative integers with $r_1+s_1+x=\frac{mt-1}{2}$. Placing a $C_4-$ factorization on $r_1$ of the $C_m[4]-$factors by Lemma \ref{lemma C4-factorization}, a $C_m-$factorization on $s_1$ of the $C_m[4]-$factors by Lemma \ref{lemma Cm-factorization} and a $\{C_4^2,C_m^2\}-$ factorization on  the remaining $x$ $C_m[4]-$factors by Lemma \ref{lemma C4^2,Cm^2-factorization} gives us a $\{C_4^{4r_1+2x+1},C_m^{4s_1+2x}\}-$ factorization of the $K_{4mt}-I$. That is, a solution to $(4, m)-$HWP$(4mt; r, s)$  exists for $r=4r_1+2x+1$ (any positive odd integer can be written in this form for non-negative $r_1$ and $x$) and $s=4s_1+2x$. It is not difficult to see that $1\leq r$ is odd and $0\leq s$ is even with $r+s=4r_1+2x+1+4s_1+2x=2mt-1=\frac{v-2}{2}$.

Therefore, a solution to $(4, m)-$HWP$(4mt; r, s)$ exists for all odd integers {\em r} and {\em t} with $r+s=2mt-1$.

\noindent \textbf{Case 2:} Now assume $t$ is even, except $t\neq 2,4$ when $m=3$.

By (\ref{four}), $K_{4mt}-I$ has a  $\{(C_m[4])^{^{(mt-2)/2}},C_4^3\}-$factorization.

Similarly, placing a $C_4-$factorization on $r_1$ of the $C_m[4]-$factors by Lemma \ref{lemma C4-factorization}, a $C_m-$ factorization on $s_1$ of the $C_m[4]-$factors by Lemma \ref{lemma Cm-factorization} and a $\{C_4^2,C_m^2\}-$ factorization on  the remaining $x$ $C_m[4]-$factors by Lemma \ref{lemma C4^2,Cm^2-factorization} gives us a\\
$\{C_4^{4r_1+2x+3},C_m^{4s_1+2x}\}-$factorization of the $K_{4mt}-I$.

Since any odd integer $r\geq 3$ can be written as $r=4r_1+2x+3$ for non-negative integers $r_1$ and $x$, we obtain that for even {\em t}, a solution to $(4, m)-$HWP$(4mt; r, s)$ exists for all odd integers $3\leq r$ (or even $0\leq s$) with $r+s=\frac{v-2}{2}$ and $m\geq 3$, except $t\neq 2,4$ and $m=3$.

For $r=1$, by the equivalence (\ref{two}) and $K_{mt}[4]\cong K_{4:mt}$,  we can write  $K_{4mt}-I\cong K_{4:mt}\oplus mtC_4$. From \cite{liu},  $K_{4:mt}$ has a $C_m-$factorization. So, placing a $C_m$-factorization on the $K_{4:mt}-$factor yields a $\{C_4^1,C_m^s\}-$factorization of $K_{4mt}-I$ with $s=2mt-2$.
\end{proof}

\section{When {\em r} is even}

Since $C_m[4]$ has no $\{C_4^1,C_m^3 \}-$factorization, we can not obtain a solution to $(4, m)-$ HWP $(4mt; r, s)$ for even {\em r} (or equivalently odd {\em s}) using the construction in the proof of Theorem \ref{thm:oddr}. However, we will use a similar construction via switching the edges of a $1-$factor from $K_4$'s with some edges of $C_m[4]$ in (\ref{three}) and (\ref{four}), then we will get a $\{C_4^2,C_m^3 \}-$factorization of the new graph.  In short, if we let $C_m^{^*}[4]\cong C_m[4]\oplus mK_4$ and $I$ is a $1-$factor  of $C_m[4]$, then we will show that
\begin{center}
$C_m^{^*}[4]-I \cong mC_4\oplus mC_4\oplus 4C_m\oplus 4C_m\oplus 4C_m$,
\end{center}
that is, $C_m^{^*}[4]-I$ has a $\{C_4^2,C_m^3\}-$factorization.

\begin{lem} \label{lemma:switch}
$(C_m[4]-I)\oplus mK_4$ has a $\{C_4^2,C_m^3\}-$factorization for some $1-$ factor $I$ in $C_m[4]$ where each $K_4$ consists of four copies of the vertex $v_i$ for any $v_i \in C_m $.
\end{lem}

\begin{proof} Consider $C_m[4]$ as the Cayley graph $\Gamma$ over $\mathbb{Z}_4\mathbb{\times Z}_m$ with connection set $\mathbb{Z}_4\times \{1,-1\}$, so each $K_4$ consists of the vertices $(0,i)$,$(1,i),(2,i),(3,i)$ for $i \in \mathbb{Z}_m$. And let $C_{(1)}=(u_0,u_1, \dots ,u_{m-1})$, $C_{(2)}=(v_0,v_1, \dots ,v_{m-1})$ and $C_{(3)}=(y_0,y_1, \dots ,y_{m-1})$ be $m-$cycles of $\Gamma$ defined by the vertices $u_i=(0,i)$, $v_i=(i^2,i)$ and $y_i=(-i^2,i)$ for $0 \leq i \leq m-2$ and $u_{m-1}=(3,m-1)$, $v_{m-1}=(1,m-1)$, $y_{m-1}=(0,m-1)$.
Then
\begin{center}
$F_1= C_{(1)} \cup (C_{(1)}+(1,0)) \cup (C_{(1)}+(2,0))\cup (C_{(1)}+(3,0))$\\
$F_2= C_{(2)} \cup (C_{(2)}+(1,0)) \cup (C_{(2)}+(2,0))\cup (C_{(2)}+(3,0))$\\
$F_3= C_{(3)} \cup (C_{(3)}+(1,0)) \cup (C_{(3)}+(2,0))\cup (C_{(3)}+(3,0))$\\
\end{center}
are $m$-cycle factors of $\Gamma$.

Let $C_{(4)}=((1,0),(2,0),(0,1),(3,1))$ and $C_{(5)}=((0,0),(1,0),(3,0),(2,0))$ be $4-$cycles of $\Gamma$. Then
\begin{center}
$F_4= \bigcup\limits_{i=0}^{m-1}(C_{(4)}+(0,i))$\\
$F_5= \bigcup\limits_{i=0}^{m-1}(C_{(5)}+(0,i))$
\end{center}
are $4$-cycle factors of $\Gamma$.
Then
\begin{center}
$\mathcal{F}=\{ F_1, F_2, F_3, F_4, F_5\}$
\end{center}
is a $2-$ factorization of $(\Gamma-I)\oplus mK_4$ where $I$ is a $1-$factor of $\Gamma$ with the edges $\{(0,i)(2,i+1)\}$ and $\{(3,i)(1,i+1)\}$ for each $ i \in \mathbb{Z}_m$.
\end{proof}

Now, we give solutions to the Hamilton-Waterloo Problem for some small cases and improve the results given in \cite{danziger}.

\begin{thm} \label{thm:v=24,m=3 exceptions}
For all positive integer $r$, a solution to $(4, 3)-${\em HWP}$(24; r, s)$ exists if and only if $r+s=11$ except possibly when $r=2$ and $r=6$.
\end{thm}

\begin{proof}
All the cases are covered by \cite{danziger} with possible exceptions when $r=2,4,6$. Let the vertex set of $K_{24}$ be $\mathbb{Z}_{24}$. Then, let\\
$F_1=(0,1,10,9)\cup (2,3,17,16)\cup (4,5,19,18)\cup (6,7,8,15)\cup (11,12,21,20)\cup (13,14,23,22),\\
F_2=(0,2,4,6)\cup (1,3,5,7)\cup (10,12,14,8)\cup (16,18,20,22)\cup (17,19,21,23)\cup (9,11,13,15),\\
F_3=(0,3,4,7)\cup (1,2,5,6)\cup (10,11,14,15)\cup (16,19,20,23)\cup (17,18,21,22)\cup (9,12,13,8),\\
F_4=(0,4,1,5)\cup (2,6,3,7)\cup (11,15,12,8)\cup (16,20,17,21)\cup (18,22,19,23)\cup (9,13,10,14),\\
F_5=(0,8,16)\cup (1,9,17)\cup (2,10,18)\cup (3,11,19)\cup (4,12,20)\cup (5,13,21)\cup (6,14,22)\cup (7,15,23),\\
F_6=(0,13,19)\cup (1,14,20)\cup (2,15,21)\cup (3,8,22)\cup (4,9,23)\cup (5,10,16)\cup (6,11,17)\cup (7,12,18),\\
F_7=(0,14,18)\cup (1,15,19)\cup (2,8,20)\cup (3,9,21)\cup (4,10,22)\cup (5,11,23)\cup (6,12,16)\cup (7,13,17),\\
F_8=(0,15,20)\cup (1,8,21)\cup (2,9,22)\cup (3,10,23)\cup (4,11,16)\cup (5,12,17)\cup (6,13,18)\cup (7,14,19),\\
F_9=(0,12,23)\cup (1,13,16)\cup (2,14,17)\cup (3,15,18)\cup (4,8,19)\cup (5,9,20)\cup (6,10,21)\cup (7,11,22),\\
F_{10}=(0,11,21)\cup (1,12,22)\cup (2,13,23)\cup (3,14,16)\cup (4,15,17)\cup (5,8,18)\cup (6,9,19)\cup (7,10,20),\\
F_{11}=(0,10,17)\cup (1,11,18)\cup (2,12,19)\cup (3,13,20)\cup (4,14,21)\cup (5,15,22)\cup (6,8,23)\cup (7,9,16)$.\\
It is easy to check that
\begin{center}
$\mathcal{F}=\{ F_1, F_2,  \dots  ,F_{11}\}$
\end{center}
is a $2-$ factorization of $K_{24}-I$ with four $C_4-$factors where
$I=\{(0,22),(1,23),$ $(2,11),(3,12),(4,13),(5,14),(6,20),(7,21),(8,17),(9,18),(10,19),(15,16)
\}$.

This completes the case $r=4$.

\end{proof}
\begin{thm} \label{thm:v=48,m=3 exceptions}
For all positive integers $r$, a solution to $(4, 3)-${\em HWP}$(48; r, s)$ exists if and only if $r+s=23$ with a possible exception when $r=6$.
\end{thm}

\begin{proof}
It is known that $(4, 3)-${\em HWP}$(48; r, s)$ has a solution with the possible exceptions when $r=6,8,10,14,$ $16,18$ \cite{danziger}. A solution to $(4, 3)-$ HWP $(12; 1, 4)$ has given in the Appendix of \cite{adams}, and by (\ref{one}), $K_{48}\cong K_{12}[4]\oplus12K_4$. Hence a $\{(C_3[4])^4, C_4[4],$ $K_{4,4},K_4\}-$ factorization of $K_{48}$ exists. Also by Lemma \ref{lemma C4-factorization}, $C_4[4]$ can be decomposed into four $C_4-$factors, by Lemma \ref{lemma:switch}, $(C_3[4]-I)\oplus 3K_4$ has a $\{C_4^2,C_3^3\}-$factorization, and it is easy to see that $K_{4,4}$ can be decomposed into two $C_4-$factors. So, we now have $8$ $C_4-$factors and $3$ $C_3-$factors already. For the remaining three $C_3[4]$'s, decompose $r_1$ of them into $C_4-$factors, $s_1$ of them into $C_3-$factors and $x$ of them into two $C_4-$factors and two $C_3-$factors where $r_1+s_1+x=3$ by Lemmas \ref{lemma C4-factorization}, \ref{lemma Cm-factorization} and \ref{lemma C4^2,Cm^2-factorization} respectively. Hence, we get   $r=4r_1+2x+8$ and $s=4s_1+2x+3$ and this gives us a $\{C_4^r,C_3^s\}-$factorization of $K_{48}-I$ for $r=8,10,12,14,16,18,20$.
\end{proof}

We would like to note that, we get the information on S. Bonvicini and M. Buratti gave solutions to all nine remaining cases of [6] independently, using clear algebraic methods in their soon to be sumbitted paper "Sharply Vertex Transitive 2-Factorizations of Cayley Graphs". But for the sake of completeness of our paper, we presented our results on six cases we have solved and left 3 cases as exceptions in our theorems.

\begin{thm} \label{thm:evenr}
For all positive even r and odd $m\geq 3$, a solution to \\$(4, m)-${\em HWP}$(v; r, s)$ exists if and only if $r+s=\frac{v-2}{2}$ except possibly $v=8m$ when $r=2$, and $v=24,48$ when $m=3$.
\end{thm}

\begin{proof}
We will consider two cases depending on the parity of {\em t}.

\noindent \textbf{Case 1:} Let {\em t} be odd.

By (\ref{three}), $K_{4mt}$ has a  $\{(C_m[4])^{^{(mt-1)/2}},K_4\}-$factorization.

Let $I$ be a $1-$factor of $K_{4m}$ as defined in Lemma \ref{lemma:switch} and, $r_1,s_1$ and $x$ be non-negative integers with $r_1+s_1+x=\frac{mt-3}{2}$. Placing a $C_4-$factorization on $r_1$ of the $C_m[4]$'s by Lemma \ref{lemma C4-factorization}, a $C_m-$factorization on $s_1$ of the $C_m[4]$'s by Lemma \ref{lemma Cm-factorization}, a $\{C_4^2,C_m^2\}-$factorization on  the $x$ of the  $C_m[4]$'s by Lemma \ref{lemma C4^2,Cm^2-factorization} and a $\{C_4^2,C_m^3\}-$factorization on the remaining $(C_m[4]-I)\oplus K_4-$factor by Lemma \ref{lemma:switch} gives us a  $\{C_4^{4r_1+2x+2},C_m^{4s_1+2x+3}\}-$factorization of the $K_{4mt}-tI$ where $tI$ gives a $1-$factor in $K_{4mt}$.

Then, since any even integer $r\geq 2$ can be written as $r=4r_1+2x+2$ for non-negative integers $r_1$ and $x$, a solution to $(4, m)-$HWP$(4mt; r, s)$ exists for any even $r\geq 2$ and odd {\em t} satisfying  $r+s=2mt-1=\frac{v-2}{2}$.

\noindent \textbf{Case 2:} Let {\em t} be even.

By (\ref{four}), $K_{4mt}$ has a  $\{(C_m[4])^{^{(mt-2/2}},K_{4,4},K_4\}-$factorization.

For $r_1+s_1+x=\frac{mt-2}{2}$, placing a $C_4-$factorization on $r_1$ of the $C_m[4]$'s, a $C_m-$factorization on $s_1$ of the $C_m[4]$'s, a $\{C_4^2$, $C_m^2\}-$factorization on  the $x$ of the  $C_m[4]$'s, two $C_4-$factors on the  $K_{4,4}-$factor and a $\{C_4^2,C_m^3\}-$factorization on the   remaining $(C_m[4]-I)\oplus K_4-$ factor yields a solution to \\ $(4, m)-$HWP$(4mt; r, s)$ for all even $r\geq 4$ except $t=2$ or $t=4$ when $m=3$.

Now, we consider the case $r=2$ and  {\em t} is even. Partitioning the vertices of $K_{4mt}$ into {\em t} sets of size $4m$ gives the equivalence: $K_{4mt}-I\cong t (K_{4m}-I^{'}) \oplus K_{4m:t}$ where $I^{'}$ is a $1-$factor of $K_{4m}$. By Case 1, $K_{4m}-I^{'}$ has a $\{C_4^2, C_m^{2m-3}\}-$factorization and  also from Theorem \ref{liu}, $K_{4m:t}$ has a $C_m-$ factorization for $t\neq 2$. Thus, $K_{4mt}-I$ has a $\{C_4^2, C_m^{2mt-3}\}-$factorization.
\end{proof}

\section{Proof of Main Result and Conclusion}

Combining the results of the previous section it is now possible to obtain the proof of the Theorem \ref{mainresult}. 

\begin{proof}[Proof of Theorem \ref{mainresult}]
Odd $r$ follows from Theorem \ref{thm:oddr} and even $r$ follows from  Theorem \ref{thm:evenr} with possible exceptions when $r=2$ and $v=8m$, and $v=24$ or $v=48$ when $m=3$. Theorem \ref{thm:v=24,m=3 exceptions} and  Theorem  \ref{thm:v=48,m=3 exceptions} cover some of these exceptions for $m=3$ and the remaining cases are $r=2$ when $v=8m$, and $v=24,48$ and $r=6$ when $m=3$.
\end{proof}

Although our solution is for odd {\em m}, our results in Lemmas are valid for even {\em m} as well and can be used in different constructions. Our result also complements the result of Fu and Huang \cite{fu}; altogether, existence of a solution to  $(4, m)-$HWP$(v; r, s)$ is shown for all integers $m\geq 3$ with a few possible exceptions. Regarding the results of Bonvicini and Buratti, only exception would be $r = 2$ when $v = 8m$ for odd $m \geq 5$. We can combine these results as follows.

\begin{thm}
For all positive integers r, s and $m\geq 3$, a solution to \\ $(4, m)-${\em HWP}$(v; r, s)$ exists if and only if $4|v$, $m|v$ and $r+s=\frac{v-2}{2}$ except  possibly when $r=2$ and $v=8m$ for $m \geq 5$ odd.
\end{thm}

\section{Acknowledgements}

This work is supported by The Scientific and Technological Research Council of Turkey (T\"UB\.{I}TAK) under project number 113F033.

\section*{References}

\bibliography{mybibfile}

\begin{thebibliography}{99}
\bibitem{adams}
P. Adams, E. J. Billington, D. E. Bryant, and S. I. El-Zanati, On the Hamilton-Waterloo Problem, {\em Graphs Combin.,} 18 (2002), 31 - 51.

\bibitem{alspach}
B. Alspach, P.J. Schellenberg, D.R. Stinson, Wagner, D., The Oberwolfach Problem and factors of uniform odd length cycles. {\em J. Comb. Theory Ser. A} 52  (1989), 20 - 43.

\bibitem{baranyai}
Z. Baranyai, Gy. R. Szasz, Hamiltonian decomposition of lexicographic product, {\em J. Comb. Theory Ser. B} 31 (1981), 253 - 261.

\bibitem{bryant}
D. Bryant, P. Danziger and M. Dean, On the Hamilton-Waterloo Problem for bipartite 2-factors, {\em J Combin Des} 21 (2013), 60 - 80.

\bibitem{buratti1}
M. Buratti, P. Danziger, A cyclic solution for an infinite class of Hamilton-Waterloo problems, Graphs and Combinatorics, accepted.

\bibitem{buratti}
M. Buratti, G. Rinaldi, On sharply vertex transitive 2-factorizations of the complete graph, {\em J. Comb. Theory Ser. A} 111 (2005) 245  256.

\bibitem{danziger}
P. Danziger, G. Quattrocchi, B. Stevens, The Hamilton-Waterloo Problem for cycle sizes 3 and 4. {\em J. Comb. Des.} 17  (2009) 342 - 352.

\bibitem{dinitz1}
J. H. Dinitz and A. C. H. Ling, The Hamilton-Waterloo Problem with triangle-factors and Hamilton cycles: the case $n \equiv 3\bmod{18}$, {\em J Combin. Math. Combin. Comput.} 70 (2009), 143 - 147.

\bibitem{dinitz2}
J. H. Dinitz and A. C. H. Ling, The Hamilton-Waterloo Problem: the case of triangle-factors and one Hamilton cycle, {\em J. Comb. Des.} 17 (2009), 160 - 176.

\bibitem{horak}
P. Horak, R. Nedela, and A. Rosa, The Hamilton-Waterloo Problem: the case of Hamilton cycles and triangle-factors, {\em Discrete Math.} 284 (2004), 181 - 188.

\bibitem{fu}
H. L. Fu, K. C. Huang, The Hamilton-Waterloo Problem for two even cycles factors, {\em Taiwanese Journal of Mathematics} 12 (2008), 933 - 940.

\bibitem{haggkvist}
R. H$\ddot{a}$ggkvist, A lemma on cycle decompositions, {\em Ann. Discrete Math.} 27 (1985), 227 - 232.

\bibitem{hoffman}
DG. Hoffman, PJ. Schellenberg, The existence of $C_k-$factorizations of $K_{2n}-I$, {\em Discrete Math.} 97 (1991), 243 - 250.

\bibitem{sibel}
M. Keranen and S. Ozkan,  The Hamilton-Waterloo Problem with 4-cycles and a single factor of n-cycles, {\em Graphs Combin.} 29 (2013), 1827 - 1837.

\bibitem{kohler}
E. K{\"o}hler, \"Uber das Oberwolfacher Problem, Beitr{\"a}ge zur Geometrischen Algebra, Basel (1977), 189-201.

\bibitem{liu}
J. Liu, The equipartite Oberwolfach Problem with uniform tables, {\em J. Combin. Theory Ser. A} 101 (2003), 20 - 34.

\bibitem{piotrowski german unpublished paper}
W. L. Piotrowski, Untersuchungen uber das Oberwolfacher Problem. Arbeitspapier (1979).

\end{thebibliography}

\bibliographystyle{amsplain}

\end{document}